\documentclass[submission, nomarks]{dmtcs-episciences}
\usepackage{graphicx}

\usepackage[utf8]{inputenc}
\usepackage[english]{babel}

\usepackage{hyperref}

\usepackage{relsize}
\usepackage{amsfonts,amsmath,amssymb,amsbsy,amsthm,enumerate}
\usepackage[autostyle]{csquotes}

\usepackage[shortlabels]{enumitem}
\usepackage{comment}
\usepackage{placeins}

\usepackage{subfigure}

\usepackage{setspace}

\usepackage{cleveref}

\usepackage{algorithm}
\usepackage[noend]{algpseudocode}
\usepackage{multicol}

\newtheorem{Theorem}{Theorem}
\newtheorem{Lemma}{Lemma}

\theoremstyle{remark}

\theoremstyle{definition}

\makeatletter
\def\BState{\State\hskip-\ALG@thistlm}
\makeatother

\title{Italian domination in generalized Petersen graphs}
\author{Hong Gao$^{1}$ \and  Jiahuan Huang$^{1}$ \and Yanan Yin$^{1}$ \and Yuansheng Yang$^{2}$}

\affiliation{$^{1}$ College of Science, Dalian Maritime University, Dalian, China\\
$^{2}$School of Computer Science and Technology, Dalian University of Technology, Dalian, China}

\begin{document}
\publicationdetails{XX}{2020}{X}{X}{XXXX}
%\received{2020-XX-XX}
%\revised{2020-XX-XX}
%\accepted{2020-XX-XX}
%\date{}

%\texttt{michael.rao@ens-lyon.fr},   \texttt{alexandre.talon@ens-lyon.fr}
\maketitle
\begin{abstract}
In a graph $G=(V,E)$, each vertex $v\in V$ is labelled with $0$, $1$ or $2$ such that each vertex labelled with $0$ is adjacent to at least one vertex labelled $2$ or two vertices labelled $1$. Such kind of labelling is called an Italian dominating function (IDF) of $G$. The weight of an IDF $f$ is $w(f)=\sum_{v\in V}f(v)$. The Italian domination number of $G$ is $\gamma_{I}(G)=\min_{f} w(f)$. Gao et al. (2019) have determined the value of $\gamma_I(P(n,3))$. In this article, we focus on the study of the Italian domination number of generalized Petersen graphs $P(n, k)$, $k\neq3$. We determine the values of $\gamma_I(P(n, 1))$, $\gamma_I(P(n, 2))$ and $\gamma_I(P(n, k))$ for $k\ge4$, $k\equiv2,3(\bmod5)$ and $n\equiv0(\bmod5)$. For other $P(n,k)$, we present a bound of $\gamma_I(P(n, k))$. With the obtained results, we partially solve the open problem presented by Bre\v{s}ar et al. (2007) by giving $P(n,1)$ is an example for which $\gamma_I=\gamma_{r2}$ and characterizing $P(n,2)$ for which $\gamma_I(P(n,2))=\gamma_{r2}(P(n,2))$. Moreover, our results imply $P(n,1)$ $(n\equiv0(\bmod\ 4))$ is Italian, $P(n,1)$ $(n\not\equiv0(\bmod\ 4))$ and $P(n,2)$ are not Italian.
\end{abstract}

\section{Introduction and notations}
We first introduce some terminologies and symbols in this paper. $G=(V,E)$ is a graph with vertex set $V$ and edge set $E$. $N(v)$ the \emph{open neighborhood} of $v\in V$ is the set of vertices which are adjacent to $v$, i.e. $N(v)=\{u|(u,v)\in E(G)\}$. $N[v]=N(v)\cup\{v\}$ is the \emph{close neighborhood} of $v$. The number of vertex in $N(v)$ is the \emph{degree} of $v$, denoted as $\deg(v)=|N(v)|$. $\Delta(G)$/$\delta(G)$ is the \emph{maximum/minimum degree} of $G$. If $\deg(v)=r$ for every $v\in V$, then $G$ is $r-$regular.

A \emph{dominating set} of $G$ is a set $D\subseteq V(G)$ and $N[D]=V(G)$.
The \emph{domination number} of $G$ is the minimum cardinality of dominating sets, denoted as $\gamma(G)$.
Domination has many variants. Roman domination \cite{Cockayne2004Roman} and rainbow domination \cite{Bresar2008rainbow} are two attractive ones.

Let $f: V\rightarrow\{0,1,2\}$ be a function on $G = (V, E)$. If for every vertex $v\in V$ with $f(v)=0$ there are at least one vertex $u$ with $f(u)=2$ in $N(v)$,     then $f$ is called a \emph{Roman dominating function} (RDF).

Let $f: V\rightarrow \{\emptyset, \{1\}, \{2\}, \{1,2\}\}$ be a function on $G = (V, E)$. If for every vertex $v\in V$ with $f(v)=\emptyset$ it holds $\cup_{u\in N(v)}f(u)=\{1,2\}$, then $f$ is called \emph{$2$-rainbow dominating function} ($2$RDF). The weight of $f$ is $w(f)=\sum_{v\in V}|f(v)|$. The \emph{2-rainbow domination number} $\gamma_{r2}(G)=\min_{f} w(f)$.

Bre\v{s}ar et al. \cite{Bresar2008rainbow} studied the 2-rainbow domination number of trees and introduced the concept of weak $\{2\}$-dominating function (W2DF).
In \cite{Chellali2016roman-2}, Chellali et al. initiated the study of Roman $\{2\}$-domination which is a generalization of Roman domination. Essentially, the Roman $\{2\}$-dominating function is the same as W2DF. This new domination is named Italian domination by Henning et al.\cite{Henning2017Italiantree}.

Let $f: V\rightarrow\{0,1,2\}$ be a function on $G = (V, E)$. If for every vertex $v\in V$ with $f(v)=0$, it holds $\sum_{u\in N(v)}f(u)\ge2$, then $f$ is called an \emph{Italian dominating function} (IDF) of $G$. The weight of $f$ is $w(f)=\sum_{v\in V}f(v)$. The \emph{Italian domination number} of $G$ is $\gamma_{I}(G)=\min_{f} w(f)$. An IDF $f$ is a $\gamma_{I}$-function if it's weight $w(f)=\gamma_{I}(G)$. If $\gamma_{I}(G)$ is equal to $2\gamma(G)$, then $G$ is called an \emph{Italian graph}.

Scholars are interested in the relationship between $\gamma_{I}(G)$ and $\gamma_{r2}(G)$, the relationship between $\gamma_{I}(G)$ and $\gamma(G)$ as well as in finding the value of $\gamma_{I}(G)$. Chellali et al. \cite{Chellali2016roman-2} show $\gamma_{I}(G)\leq \gamma_{r2}(G)$ for a graph $G$.
Bre\v{s}ar et al. \cite{Bresar20072rainbowdomination} present an open problem that is for which classes of graphs $\gamma_{r2}=\gamma_{I}$. And they prove for trees $\gamma_{I}(T)=\gamma_{r2}(T)$ \cite{Bresar2008rainbow}.
St\c{e}pie\'{n} et al. \cite{Stepien20152rainbowcnc5} show $\gamma_{I}(C_{n}\Box C_5)=\gamma_{r2}(C_{n}\Box C_5)$.
If $\gamma_{I}(G)=2\gamma(G)$, then $G$ is called an Italian graph. Stars and double stars are Italian graphs \cite{Henning2017Italiantree}.
More studies on finding the Italian domination number, please refer to \cite{Li2018weak-2, Gao2019ItalianCnPm}. The value of $\gamma_{I}(P(n,3))$ is determined by Gao et al. \cite{Gao2019pn3}.
There are many studies \cite{Hao2018globalItalian, Rahmouni, FanW, Haynes2019PerfectItalian} are related to Italian domination.

In this article, we determine the values of $\gamma_{I}(P(n,1))$, $\gamma_{I}(P(n,2))$ and $\gamma_{I}(P(n,k))$ for $k\ge 4$, $k\equiv2,3(\bmod5)$ and $n\equiv0(\bmod5)$.
We give $P(n,1)$ is an example for which $\gamma_I=\gamma_{r2}$ and we characterize $P(n,2)$ for which $\gamma_I(P(n,2))=\gamma_{r2}(P(n,2))$. Our results imply $P(n,1)$ $(n\equiv0(\bmod\ 4))$ is Italian, $P(n,1)$ $(n\not\equiv0(\bmod\ 4))$ and $P(n,2)$ are not Italian.

\section{The Italian domination number of $P(n,1)$}
Let $G$ be $P(n,k)$, we use the notation $f$ to define an Italian domination function on $G$,
\begin{eqnarray*}
f(V(G))= \left(\begin{array}{lllllll}
f(v_0)  & f(v_2)    & f(v_4)    & \cdots& f(v_{2n-2})\\
f(v_1)  & f(v_3)    & f(v_5)    & \cdots& f(v_{2n-1})\\
\end{array}\right).
\end{eqnarray*}

Figure \ref{fig:pnk} $(a)$ shows the graph $P(6,2)$ and $(b)$ shows $P(6,2)$ under $f$, in which red vertices represent $f(v)=1$, blue vertice represent $f(v)=0$.
\begin{figure}[htbp]
\centering
\includegraphics[scale=1.0]{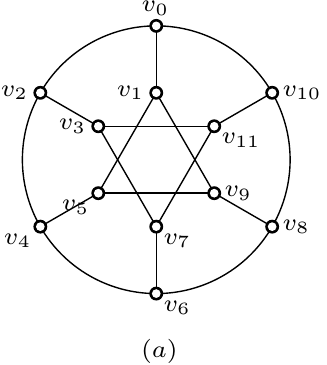}
\includegraphics[scale=1.0]{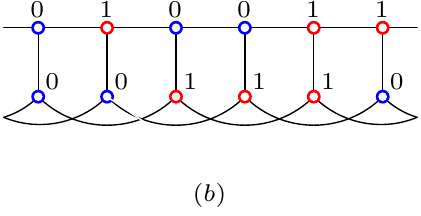}
\caption{$(a)$ Graph $P(6,2)$. $(b)$ $P(6,2)$ under $f$.}
\label{fig:pnk}
\end{figure}

\begin{Theorem}
$\gamma_{I}(P(n,1))\leq n$.
\label{thm:lowerboundpn1}
\end{Theorem}
\begin{proof}
Let
\begin{equation*}
f=\left\{
\begin{array}{llll}
{\left(\begin{array}{lllll}
1 &0 \\
0 &1 \\
\end{array}\right)}^\frac{n}{2}       &n\equiv0(\bmod\ 2),\\
{\left(\begin{array}{lllll}
1 &0 \\
0 &1 \\
\end{array}\right)}^\frac{n-1}{2}{\left(\begin{array}{lllll}
1  \\
0  \\
\end{array}\right)}       &n\equiv1(\bmod\ 2),\\
\end{array}\right.
\end{equation*}
\noindent where $\frac{n}{2}(\frac{n-1}{2})$ means we repeat the two columns $\frac{n}{2}(\frac{n-1}{2})$ times. Then
$w(f)=n$, thus $\gamma_{I}(P(n,1))\leq n$.
\end{proof}

Next we will prove $\gamma_{I}(P(n,1))\ge n$. Let $f$ be an IDF on $P(n,1)$, we denote $V^{i}=\{v_{2i}, v_{2i+1}\}$, $w(f_i)=f(v_{2i})+f(v_{2i+1})$ $(0\leq i\leq n-1)$.

\begin{Lemma}
If $w(f_i)=0$ $(0\leq i\leq n-1)$, then $w(f_{i-1})+w(f_{i+1})\ge 4$, where subscripts are read modulo 2n.
\label{lem:44444}
\end{Lemma}
\begin{proof}
$w(f_i)=0$, i.e. $f(v_{2i})=f(v_{2i+1})=0$, by the definition of IDF, $f(v_{2i-2})+f(v_{2i+2})\ge 2$ and $f(v_{2i-1})+f(v_{2i+3})\ge 2$, i.e. $f(v_{2i-2})+f(v_{2i-1})+f(v_{2i+2})+f(v_{2i+3})=w(f_{i-1})+w(f_{i+1})\ge 4$.
\end{proof}

\begin{Theorem}
$\gamma_{I}(P(n,1)) \geq n.$
\label{thm:upperboundpn1}
\end{Theorem}

\begin{proof}
Let $G$ be $P(n,1)$, by bagging approach \cite{Gao2019ItalianCnPm}, specifically, we use the following steps to put vertices into five bags $B_{1},\cdots, B_{5}$.

Initialization: $m_1=m_2=m_3=m_4=m_5=0$, $B_{1}=B_{2}=B_{3}=B_{4}=B_{5}=\emptyset$,
$D[i]=0$ for $i=0$ up to $n-1$.

Step 1. For every $i$ with $w(f_{i})=0\wedge w(f_{i+1})=2$ do
$m_1=m_1+1$, $D[i]=D[i+1]=1$, $B_{1}=B_{1}\cup V^{i}\cup V^{i+1}$.

Step 2. For every $i$ with $w(f_{i})=0\wedge w(f_{i+1})\geq3\wedge w(f_{i+2})=0\wedge D[i+2]=0$
do $m_2=m_2+1$, $D[i]=D[i+1]=D[i+2]=1$, $B_{2}=B_{2}\cup V^{i}\cup V^{i+1}\cup V^{i+2}$.

Step 3. For every $i$ with $w(f_{i})=0\wedge w(f_{i+1})\geq3\wedge (w(f_{i+2})\ge1\vee D[i+2]=1)$
do $m_3=m_3+1$, $D[i]=D[i+1]=1$, $B_{3}=B_{3}\cup V^{i}\cup V^{i+1}$.

Step 4. For every $i$ with $w(f_{i})=0\wedge w(f_{i+1})\leq1$, by Lemma \ref{lem:44444}, $w(f_{i-1})\ge3$ do
$m_4=m_4+1$, $D[i-1]=D[i]=1$, $B_{4}=B_{4}\cup V^{i-1}\cup V^{i}$.

By now for every $i$ with $D[i]=0$, $w(f_{i})>0$.

Let $B_{5}=V(G)-B_{1}-B_{2}-B_{3}-B_{4}$, $m_5=\frac{|B_{5}|}{2}$.

Since $2m_1+3m_2+2m_3+2m_4+m_5=n$,
\begin{equation*}
\begin{split}
f(V(G))&\geq m_1\times 2+m_2\times 3+m_3\times 3+m_4\times 3+m_5\\
&=2m_1+3m_2+3m_3+3m_4+m_5\\
&\geq 2m_1+3m_2+2m_3+2m_4+m_5\\
&=n.
\end{split}
\end{equation*}
Thus, $\gamma_{I}(G) \geq n.$
\end{proof}

By Theorem \ref{thm:lowerboundpn1} and Theorem \ref{thm:upperboundpn1}, we obtain the value of $\gamma_{I}(P(n,1))$.
\begin{Theorem}
For any integer $n\geq3$, $\gamma_{I}(P(n,1)) =n$.
\label{thm:Italianpn1}
\end{Theorem}

\begin{Theorem}\emph{(\cite{shao20192rainbowpnk})}
If $n\geq5$, we have $\gamma_{r2}(P(n,1))=n$.
\label{thm:2rainbowpn1}
\end{Theorem}

By Theorem \ref{thm:Italianpn1}-\ref{thm:2rainbowpn1}, $P(n,1)$ is an example for which $\gamma_{I}=\gamma_{r2}$.

\begin{Theorem}\emph{(\cite{Ebrahimi2009dominationpnk})}
If $n\geq3$, then we have
\begin{equation*}
\gamma(P(n, 1))=
\begin{cases}
\frac{n}{2}+1           &if\ n\equiv2(\bmod\ 4),\\
\lceil\frac{n}{2}\rceil &otherwise.
\end{cases}\notag
\end{equation*}
\label{thm:dominationpn1}
\end{Theorem}

By Theorem \ref{thm:Italianpn1} and Theorem \ref{thm:dominationpn1}, $\gamma_{I}(P(n, 1))=2\gamma(P(n, 1))$ for $n\equiv0(\bmod\ 4)$. Hence,  $P(n,1)$ for $n\equiv0(\bmod\ 4)$ is Italian, while $P(n,1)$ for $n\not\equiv0(\bmod\ 4)$ is not Italian.

\section{The Italian domination number of $P(n,2)$}
\begin{Lemma}\emph{(\cite{Chellali2016roman-2})}
For every graph $G$, $\gamma_{I}(G)\leq \gamma_{r2}(G)$, where $\gamma_{r2}(G)$ is the 2-rainbow domination number.
\label{le2}
\end{Lemma}

\begin{Lemma}\emph{(\cite{tong2009pn2})}
\begin{equation*}
\gamma_{r2}(P(n, 2))=
\begin{cases}
\lceil\frac{4n}{5}\rceil     &n\equiv 0, 3, 4, 9(\bmod\ 10),\\
\lceil\frac{4n}{5}\rceil+1   &n\equiv 1, 2, 5, 6, 7, 8(\bmod\ 10).
\end{cases}\notag
\label{lem:tong2rainbowpn2}
\end{equation*}
\end{Lemma}

\begin{Theorem}
\begin{equation*}
\gamma_{I}(P(n, 2)) \leq
\begin{cases}
\lceil\frac{4n}{5}\rceil     &n\equiv 0, 3, 4(\bmod\ 5),\\
\lceil\frac{4n}{5}\rceil+1   &n\equiv 1, 2(\bmod\ 5).
\end{cases}\notag
\end{equation*}
\label{thm:upperboundpn2}
\end{Theorem}
\begin{proof}
By Lemma \ref{le2}-\ref{lem:tong2rainbowpn2},
\begin{align}
\notag &\gamma_{I}(P(n, 2)) \leq\gamma_{r2}(P(n, 2))=\lceil\frac{4n}{5}\rceil   &n&\equiv 0, 3, 4, 9(\bmod\ 10),\\
\notag &\gamma_{I}(P(n, 2)) \leq\gamma_{r2}(P(n, 2))=\lceil\frac{4n}{5}\rceil+1 &n&\equiv 1, 2, 6, 7(\bmod\ 10).
\end{align}

For $n\equiv 5, 8(\bmod\ 10)$, we can get the upper bound is $\lceil\frac{4n}{5}\rceil$ instead of $\lceil\frac{4n}{5}\rceil+1$ by constructing an IDF $f$ as follows.

\begin{equation*}
f=\left\{
\begin{array}{llll}
{\left(\begin{array}{lllll}
0 &1 &0  &0  &1  \\
0 &0 &1  &1  &0  \\
\end{array}\right)}^\frac{n}{5}       &n \equiv5(\bmod\ 10),\\
{\left(\begin{array}{lllll}
0 &1 &0  &0  &1  \\
0 &0 &1  &1  &0  \\
\end{array}\right)}^\frac{n-3}{5}{\left(\begin{array}{lllll}
0  &0  &1\\
1  &1  &0\\
\end{array}\right)}        & n \equiv8(\bmod\ 10).\\
\end{array}\right.
\end{equation*}

Then $w(f)=\frac{n}{5}\times 4=\frac{4n}{5}$ for $n \equiv5(\bmod\ 10)$, $w(f)=\frac{n-3}{5}\times 4+3=\frac{4n+3}{5}$ for $n \equiv8(\bmod\ 10)$. So, $\gamma_{I}(P(n,2)) \leq \lceil\frac{4n}{5}\rceil$ for $n\equiv 5, 8(\bmod\ 10)$.

Thus,
\begin{equation*}
\gamma_{I}(P(n, 2)) \leq
\begin{cases}
\lceil\frac{4n}{5}\rceil     &n\equiv 0, 3, 4(\bmod\ 5),\\
\lceil\frac{4n}{5}\rceil+1   &n\equiv 1, 2(\bmod\ 5).
\end{cases}\notag
\end{equation*}
\end{proof}

Next we will prove the lower bound of $\gamma_{I}(P(n, 2))$ is equal to the upper bound.

\begin{Lemma}\emph{(\cite{Chellali2016roman-2})}
If $G$ is a connected graph and maximum degree $\Delta(G)=\Delta$, then $\gamma_{I}(G)\geq \frac{2|V(G)|}{\Delta+2}$.
\label{lem:pn2low}
\end{Lemma}

\begin{Theorem} Let $G=P(n, 2)$, then
\begin{equation*}
\gamma_{I}(G) \ge
\begin{cases}
\lceil\frac{4n}{5}\rceil     &n\equiv 0, 3, 4(\bmod\ 5),\\
\lceil\frac{4n}{5}\rceil+1   &n\equiv 1, 2(\bmod\ 5).
\end{cases}\notag
\end{equation*}
\label{thm:lowerboundpn2}
\end{Theorem}

\begin{proof}
Case 1. $n\equiv 0, 3, 4 \pmod {5}$. In $G=P(n,2)$, $\Delta(G)=3$, $|V(G)|=2n$, by Lemma \ref{lem:pn2low},
$\gamma_{I}(P(n,2)) \geq\big\lceil\frac{4n}{5}\big\rceil.$

Case 2. $n\equiv 1, 2 \pmod {5}$.
Let $f=(V_{0}, V_{1}, V_{2})$ be a $\gamma_{I}$-function of $P(n,2)$, where $V_{i}=\{v\in V|f(v)=i\}$, then $\gamma_{I}(P(n,2))=w(f)=\sum_{v \in V}f(v)=|V_1|+2|V_2|$. First, we prove $\gamma_{I}(P(n,2))\ge \lceil\frac{4n}{5}\rceil+1$, i.e. $\gamma_{I}(P(n,2)) >\lceil\frac{4n}{5}\rceil$ is equivalent to $r_{f}(V) >0.2$ for $n\equiv 1\pmod {5}$ and $r_{f}(V) >0.4$ for $n\equiv 2\pmod {5}$, where

\begin{align}
\notag r_{f}(V)&=\sum\limits_{v\in V}r_{f}(v)=\sum\limits_{v\in V}(g_{f}(v)-0.4)=g_{f}(V)-0.8n,\\
\notag g_{f}(v) &=
\begin{cases}
0.2|V_{1}\cap N(v)| +0.5|V_{2}\cap N(v)|     &v\in V_{0},\\
0.4+0.2|V_{1}\cap N(v)| +0.5|V_{2}\cap N(v)| &v\in V_{1},\\
0.5+0.2|V_{1}\cap N(v)| +0.5|V_{2}\cap N(v)| &v\in V_{2}.\\
\end{cases}
\end{align}

In fact,
\begin{equation*}
\begin{array}{lllllll}
g_{f}(V)&=& \sum\limits_{v\in V}g_{f}(v)\\
&=& \sum\limits_{v\in V_{0}}g_{f}(v) +  \sum\limits_{v\in V_{1}}g_{f}(v) +  \sum\limits_{v\in V_{2}}g_{f}(v)\\
& = & \sum\limits_{v\in V_{0}}(0.2|V_{1}\cap N(v)|+ 0.5|V_{2}\cap N(v)|)\\
  &&+\sum\limits_{v\in V_{1}}(0.4+0.2|V_{1}\cap N(v)|+ 0.5|V_{2}\cap N(v)|)\\
&& +\sum\limits_{v\in V_{2}}(0.5+0.2|V_{1}\cap N(v)|+ 0.5|V_{2}\cap N(v)|)\\
& = &0.2\sum\limits_{v\in V_{1}}|N(v)\cap V_{0}|+ 0.5\sum\limits_{v\in V_{2}}|N(v)\cap V_{0}|\\
  &&+0.4|V_{1}|+0.2\sum\limits_{v\in V_{1}}|N(v)\cap V_{1}|+ 0.5\sum\limits_{v\in V_{2}}|N(v)\cap V_{1}|\\
  &&+0.5|V_{2}|+0.2\sum\limits_{v\in V_{1}}|N(v)\cap V_{2}| + 0.5\sum\limits_{v\in V_{2}}|N(v)\cap V_{2}|\\
& = & 0.4|V_{1}| + 0.2\sum\limits_{v\in V_{1}}|N(v)\cap V_{0}|+ 0.2\sum\limits_{v\in V_{1}}|N(v)\cap V_{1}| +0.2\sum\limits_{v\in V_{1}}|N(v)\cap V_{2}| \\
   &&+ 0.5|V_{2}| + 0.5\sum\limits_{v\in V_{2}}|N(v)\cap V_{0}|+ 0.5\sum\limits_{v\in V_{2}}|N(v)\cap V_{1}| +0.5\sum\limits_{v\in V_{2}}|N(v)\cap V_{2}| \\
& = & 0.4|V_{1}|+0.2\sum\limits_{v\in V_{1}}|N(v)|+ 0.5|V_{2}| + 0.5\sum\limits_{v\in V_{2}}|N(v)|\\
& = & 0.4|V_{1}|+0.2\times3|V_{1}|+ 0.5|V_{2}| + 0.5\times3|V_{2}|=|V_{1}|+2|V_{2}|=\gamma_{I}(P(n,2)).
\end{array}
\end{equation*}

Then $r_{f}(V)=g_{f}(V)-0.8n=\gamma_{I}(P(n,2))-0.8n$ i.e. $\gamma_{I}(P(n,2))=r_{f}(V)+0.8n$. Thus,
$\gamma_{I}(P(n,2)) >\lceil\frac{4n}{5}\rceil$ is equivalent to $r_{f}(V)>\big\lceil\frac{4n}{5}\big\rceil-0.8n$.

For $n\equiv 1(\bmod\ 5)$, let $n=5q+1$, then  $\big\lceil\frac{4n}{5}\big\rceil-0.8n=\big\lceil\frac{4(5q+1)}{5}\big\rceil-\frac{4(5q+1)}{5}=4q+1-4-\frac{4}{5}=0.2$.

For $n\equiv 2(\bmod\ 5)$, let $n=5q+2$, then $\big\lceil\frac{4n}{5}\big\rceil-0.8n=\big\lceil\frac{4(5q+2)}{5}\big\rceil-\frac{4(5q+2)}{5}=4q+2-4q-\frac{8}{5}=0.4$.

Hence, $\gamma_{I}(P(n,2)) >\lceil\frac{4n}{5}\rceil$ is equivalent to $r_{f}(V) >0.2$ for $n\equiv 1\pmod {5}$ and $r_{f}(V) >0.4$ for $n\equiv 2\pmod {5}$.

Then, we will prove $r_{f}(V) >0.2$ for $n\equiv 1\pmod {5}$ and $r_{f}(V) >0.4$ for $n\equiv 2\pmod {5}$.

Let $E_{11}=\{(u,v)\in E|u, v\in V_1\}$, $E_{12}=\{(u,v)\in E|u\in V_1, v\in V_2\}$, then we have some findings shown in Table \ref{tab:findings}. 

\begin{table}[!h]
\caption{Findings}\label{tab:findings}
\begin{center}
\begin{tabular}{|p{3cm}|p{3cm}|p{3cm}| p{3cm}|}
\hline
Finding 1 & Finding 2 & Finding 3 & Finding 4\\
&&&\\
For each vertex $v\in V$, $g_{f}(v)\geq0.4$.&
\begin{minipage}{0.2\textwidth}
     \includegraphics[scale=1.0]{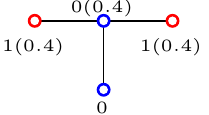}
\end{minipage}
&
\begin{minipage}{0.2\textwidth}
     \includegraphics[scale=1.0]{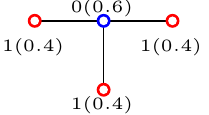}
\end{minipage}
&
\begin{minipage}{0.2\textwidth}
     \includegraphics[scale=1.0]{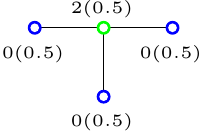}
\end{minipage}
 \\
&The numbers outside the brackets are $f(v)$ and the numbers inside the brackets are $g_{f}(v)$.&
If $\exists$ $v\in V_0$ with $|N(v)\cap V_1|=3$, then $r_{f}(V)\geq 0.2$.&
If $\exists$ $v\in V_2$, then $r_{f}(V)\geq 0.4$.\\
&
If $\exists$ $v\in V_0$ with $|N(v)\cap V_1|=2$, then $r_{f}(V)\geq 0$.&
& \\ \hline
Finding 5 & Finding 6 & Finding 7 & Finding 8\\
&&&\\
\begin{minipage}{0.2\textwidth}
     \includegraphics[scale=1.0]{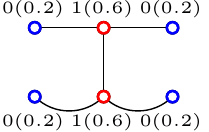}
\end{minipage}
&
\begin{minipage}{0.2\textwidth}
     \includegraphics[scale=1.0]{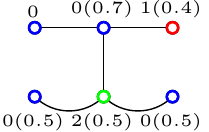}
\end{minipage}
&
\begin{minipage}{0.2\textwidth}
     \includegraphics[scale=1.0]{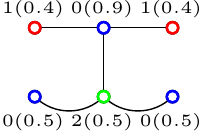}
\end{minipage}
&
\begin{minipage}{0.2\textwidth}
     \includegraphics[scale=1.0]{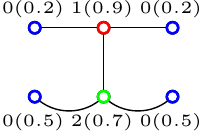}
\end{minipage}
\\
&&&\\
If $\exists$ $(u,v)\in E_{11}$, then $r_{f}(V)\geq 0.4$.&
If $\exists$ $v\in V_0$ with $|N(v)\cap V_1|=|N(v)\cap V_2|=1$, then $r_{f}(V)\geq 0.6$.&
If $\exists$ $v\in V_0$ with $|N(v)\cap V_1|=2$, $|N(v)\cap V_2|=1$, then $r_{f}(V)\geq 0.8$.&
If $\exists$ $(u,v)\in E_{12}$, then $r_{f}(V)\geq 1$.
\\ \hline
\end{tabular}
\end{center}
\end{table}

By Finding 4-Finding 8, in the following cases (Case 1-Case 5) $r_{f}(V)>0.4$.

Case 1. $|V_2|\geq 2$.

Case 2. $|E_{11}|\geq 2$.

Case 3. $|V_2|+|E_{11}|\geq 2$.

Case 4. $|E_{12}|\geq 1$.

Case 5. $\exists$ $v\in V_0$ with $|N(v)\cap V_1|\ge1$, $|N(v)\cap V_2|=1$.

Besides Case 1-Case 5, there are four cases.

Case 6. $|V_2|=1$.

Case 7. $|E_{11}|=1$.

Case 8. $\exists$ $v\in V_0$ with $|N(v)\cap V_1|=3$.

Case 9. Excluding all of above cases, then $|V_2|=0$, $|E_{11}|=0$ and $|N(v)\cap V_1|\leq 2$ for $v\in V_0$. According to IDF definition, there must exist $v_{i}$ with $f(v_{i})=1$ for $i(\bmod2)=0$.

Next, we will prove in Case 6-Case 9 $r_{f}(V) >0.2$ for $n\equiv 1\pmod {5}$ and $r_{f}(V) >0.4$ for $n\equiv 2\pmod {5}$.

Case 6. $|V_2|=1$. Let $f(v_{0})=2$ or $f(v_{1})=2$.
For $n\equiv 1\pmod {5}$, by Finding 4, $r_{f}(V)\ge 0.4>0.2$.
For $n\equiv 2\pmod {5}$, by contrary, suppose $r_{f}(V)\leq 0.4$.

If $f(v_{0})=2$, by excluding Case 1 and Case 4, $f(v_{1})=f(v_2)=0$. By excluding  Case 1 and Case 5, $f(v_3)=f(v_4)=0$. Then by the definition of IDF, $f(v_5)=1$.
Consequently, $v_1\in V_0$ and $|N(v_1)\cap V_1|=|N(v_1)\cap V_2|=1$, by Finding 6, $r_{f}(V)\ge 0.6>0.4$. This contradicts the hypothesis.

If $f(v_{1})=2$, by excluding Case 1, Case 4, and Case 5, $f(v_{0})=f(v_{5})=0$, $f(v_{2})=0$. By the definition of IDF,  $f(v_{3})=f(v_{4})=1$.
Consequently, $v_5\in V_0$ and $|N(v_5)\cap V_1|=|N(v_5)\cap V_2|=1$, by Finding 6, $r_{f}(V)\ge 0.6> 0.4$, a contradiction.

The IDF $f$ in this case is shown in Figure \ref{01}, where green, red and blue vertices stand for $f(v)=2,1,0$, respectively.
\begin{figure}[!ht]
\centering
\includegraphics[scale=1.0]{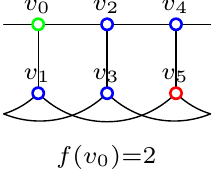}
\hspace{10bp}
\includegraphics[scale=1.0]{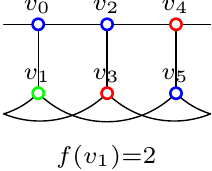}
\caption{$f$ in Case 6.}\label{01}
\end{figure}

Case 7. $|E_{11}|=1$. For $n\equiv 1\pmod {5}$, by Finding 5, $r_{f}(V)\ge 0.4> 0.2$.
For $n\equiv 2\pmod {5}$, by contrary, suppose $r_{f}(V)\leq 0.4$. Let $f(v_0)=f(v_1)=1$ or $f(v_0)=f(v_2)=1$ or $f(v_1)=f(v_5)=1$.

If $f(v_0)=f(v_1)=1$, by excluding Case 2 and Case 3, $f(v_2)=f(v_5)=0$. Then $f(v_4)=1$.
After that, $f(v_3)=1$ or $0$. If $f(v_3)=1$, then $|N(v_2)\cap V_1|=3$, by Finding 3 and Finding 5, $r_{f}(V)\ge 0.6> 0.4$, a contradiction. Thus, $f(v_3)=0$. By excluding Case 2, $f(v_6)=0$. Then $f(v_7)=1$. After that, $f(v_8)=1$ or $0$. If $f(v_8)=1$, then $|N(v_6)\cap V_1|=3$, by Finding 3 and Finding 5, $r_{f}(V)\ge 0.6> 0.4$, a contradiction. So, $f(v_8)=0$. By the definition of IDF, $f(v_9)=1$. Then, $f(v_5)=0$ and $|N(v_5\cap V_1|=3$, by Finding 3 and Finding 5, $r_{f}(V)\ge 0.6> 0.4$, a contradiction.

If $f(v_0)=f(v_2)=1$,
by excluding Case 2 and Case 3, $f(v_1)=f(v_3)=f(v_4)=0$. Then $f(v_5)=1$.
After that, $f(v_{6})=1$ or $0$. If $f(v_6)=1$, then $|N(v_4)\cap V_1|=3$, by Finding 3 and Finding 5, $r_{f}(V)\ge0.6> 0.4$, a contradiction. So, $f(v_{6})=0$. By the definition of IDF, $f(v_7)=f(v_8)=1$. By excluding Case 2, $f(v_9)=f(v_{10})=f(v_{11})=0$. Then $f(v_{12})=1$.
By excluding Case 2, $f(v_{13})=f(v_{14})=0$. Then $f(v_{15})=1$.
After that, $f(v_{16})=1$ or $0$. If $f(v_{16})=1$, then $|N(v_{14})\cap V_1|=3$, by Finding 3 and Finding 5, $r_{f}(V)\ge0.6> 0.4$, a contradiction. So, $f(v_{16})=0$. By IDF definition, $f(v_{17})=f(v_{18})=1$. By excluding Case 2, $f(v_{19})=f(v_{20})=f(v_{21})=0$.
Continue in this way, it has $f(v_i)=1$, $i\equiv2,5,7,8(\bmod\ 10)$ and $f(v_i)=0$, $i\equiv3,4,6,9,10,11(\bmod\ 10)$.
Since $n=5q+2$, $2n-2=10q+2$, then $v_{2n-2}\in V_1$, $v_{0}\in V_{1}$, and $(v_{2n-2},v_{0})\in E$, by Finding 5, $r_{f}(V)\ge0.8>0.4$, a contradiction.

If $f(v_1)=f(v_5)=1$,
By excluding Case 2 and Case 3, $f(v_0)=f(v_4)=f(v_9)=0$. Then $f(v_2)=1$. By excluding Case 2, $f(v_3)=0$.
After that, $f(v_6)=0$ or $1$. If $f(v_6)=1$, then $v_4\in V_0$ and $|N(v_4)\cap V_1|=3$, by Finding 3 and Finding 5, $r_{f}(V)\ge0.6>0.4$, a contradiction. So, $f(v_{6})=0$. By IDF definition, $f(v_7)=f(v_8)=1$.
By excluding Case 2, $f(v_{10})=f(v_{11})=0$. Then $f(v_{12})=1$.
By excluding Case 2, $f(v_{13})=f(v_{14})=0$. Then, $f(v_{15})=1$.
After that, $f(v_{16})=0$ or $1$. If $f(v_{16})=1$, then $|N(v_{14})\cap V_1|=3$, by Finding 3 and Finding 5, $r_{f}(V)\ge0.6>0.4$, a contradiction. So, $f(v_{16})=0$. According to IDF definition, $f(v_{17})=f(v_{18})=1$.
By excluding Case 2, $f(v_{19})=f(v_{20})=f(v_{21})=0$. Then $f(v_{22})=1$.
Continue in this way, it has $f(v_i)=1$, $i\equiv2,5,7,8(\bmod\ 10)$ and $f(v_i)=0$, $i\equiv3,4,6,9,10,11(\bmod\ 10)$.
Since $n=5q+2$, $2n-2=10q+2$, then $v_{2n-2}\in V_1$, $v_{0}\in V_{0}$ and $|N(v_{0})\cap V_1|=3$, by Finding 3 and Finding 5, $r_{f}(V)\ge0.6>0.4$, a contradiction.

The IDF $f$ in this case is shown with Figure \ref{1022}.

\begin{figure}[!ht]
%\leftflush
\centering
\includegraphics[scale=1.0]{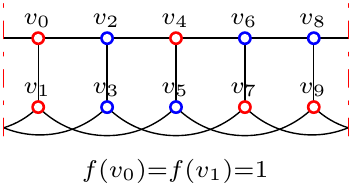}\hspace{5bp}
\includegraphics[scale=1.0]{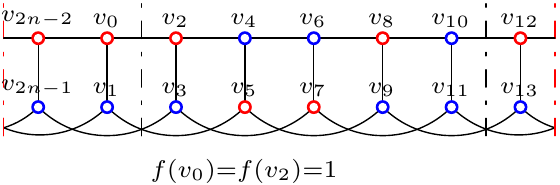}\\
\includegraphics[scale=1.0]{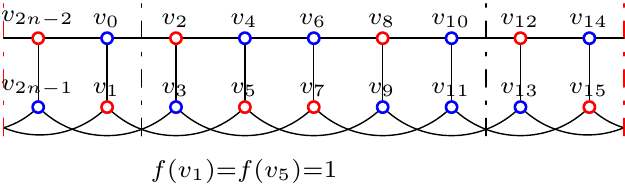}\hspace{5bp}
\caption{$f$ in Case 7.} \label{1022}
\end{figure}

Case 8. $\exists$ $v\in V_0$ with $|N(v)\cap V_1|=3$.
By contrary, suppose $r_{f}(V)\leq 0.2$ for $n\equiv1(\bmod5)$ and $r_{f}(V)\leq 0.4$ for $n\equiv2(\bmod5)$. Let $f(v_{2})=0$, $f(v_{0})=f(v_{3})=f(v_{4})=1$.
By excluding Case 6 and Case 7, $f(v_1)=f(v_5)=f(v_6)=f(v_7)=0$. By the definition of IDF, $f(v_8)=f(v_9)=1$.
Then, $(v_{8},v_{9})\in E$, by Finding 3 and Finding 5, $r_{f}(V)\ge0.6>0.4>0.2$, a contradiction. The $f$ is shown with Figure \ref{359}.
\begin{figure}[!ht]
%\leftflush
\centering
\includegraphics[scale=1.0]{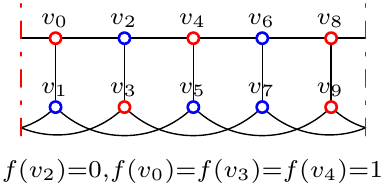}
\caption{$f$ in Case 8.} \label{359}
\end{figure}

Case 9. $\exists$ $v_{i}$ with $f(v_{i})=1$ and $i(\bmod2)=0$.
By contrary, suppose $r_{f}(V)\leq 0.2$ for $n\equiv1(\bmod5)$ and $r_{f}(V)\leq 0.4$ for $n\equiv2(\bmod5)$.
Let $f(v_{0})=1$.
By excluding Case 6 and Case 7, $f(v_1)=f(v_2)=0$. Then $f(v_3)=1$ or $0$.

If $f(v_{3})=1$.
By excluding Case 8, $f(v_4)=0$.
According to IDF definition, $f(v_5)=f(v_6)=1$.
By excluding Case 7, $f(v_7)=f(v_8)=f(v_9)=0$. Then $f(v_{10})=1$.
By excluding Case 7, $f(v_{11})=f(v_{12})=0$. Then $f(v_{13})=1$.
By excluding Case 8, $f(v_{14})=0$. Then $f(v_{15})=f(v_{16})=1$.
By excluding Case 7, $f(v_{17})=f(v_{18})=f(v_{19})=0$.
By excluding Case 8, $f(v_{21})=0$.
According to IDF definition, $f(v_{20})=1$.
Continue in this way, it has $f(v_i)=1$, $i\equiv3,5,6,10(\bmod\ 10)$ and
$f(v_i)=0$, $i\equiv2,4,7,8,9,11(\bmod\ 10)$.
For $n=5q+1$, $2n-2=10q$, $f(v_{2n-2})=1$, then $v_{2n-2}, v_{0}\in V_{1}$ and $(v_{2n-2},v_{0})\in E$. For $n=5q+2$, $2n-1=10q+3$, $f(v_{2n-1})=1$, then $v_{3},v_{2n-1}\in V_1$ and $(v_{3},v_{2n-1})\in E$. By Case 7, $r_{f}(V)>0.2$ for $n\equiv1(\bmod5)$ and $r_{f}(V)>0.4$ for $n\equiv2(\bmod5)$, a contradiction. The IDF $f$ is shown in Figure \ref{1020}.

If $f(v_{3})=0$.
By definition of IDF, $f(v_{4})=1$.
By excluding Case 7, $f(v_{5})=f(v_{6})=0$. Then $f(v_{7})=1$.
By excluding Case 8, $f(v_{8})=0$. Then $f(v_{9})=f(v_{10})=1$.
By excluding Case 7, $f(v_{11})=f(v_{12})=f(v_{13})=0$. Then $f(v_{14})=1$.
By excluding Case 8, $f(v_{15})=f(v_{16})=0$. Then $f(v_{17})=1$.
By excluding Case 8, $f(v_{18})=0$. Then $f(v_{19})=f(v_{20})=1$.
By excluding Case 7, $f(v_{21})=0$.
Continue in this way, it has $f(v_i)=1$, $i\equiv4,7,9,10(\bmod\ 10)$ and $f(v_i)=0$, $i\equiv2,3,5,6,8,11(\bmod\ 10)$.
For $n=5q+1$, $2n-2=10q$, $f(2n-2)=1$, then $v_{2n-2}, v_0\in V_1$ and $(v_{2n-2},v_0)\in E$, by Case 7, $r_{f}(V)>0.2$, a contradiction.
For $n=5q+2$, $2n-1=10q+3$, $f(2n-1)=0$, then $N(v_{3})\cap V_0=2$, this case can not meet the definition of IDF. So, $r_{f}(V)\leq0.4$ is not true. The $f$ is shown with Figure \ref{1020}.

\begin{figure}[!ht]
\centering
\includegraphics[scale=1.0]{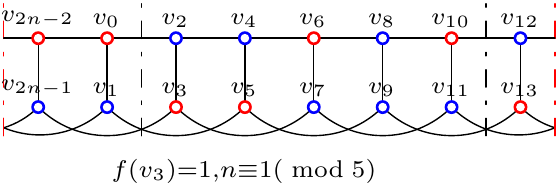}\hspace{5bp}
\includegraphics[scale=1.0]{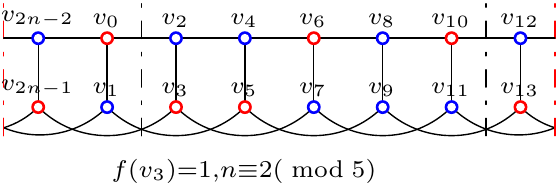}\\
\includegraphics[scale=1.0]{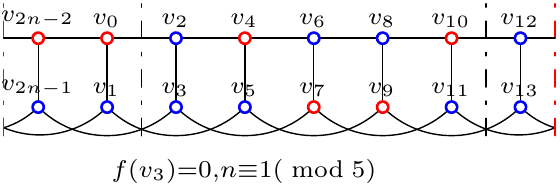}\hspace{5bp}
\includegraphics[scale=1.0]{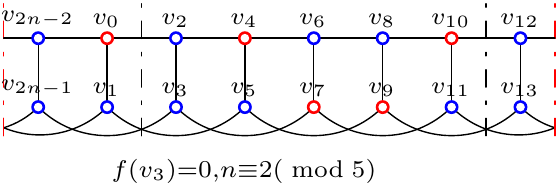}
\caption{$f$ in Case 9.} \label{1020}
\end{figure}

\end{proof}

By Theorem \ref{thm:upperboundpn2} and Theorem \ref{thm:lowerboundpn2}, we have
\begin{Theorem}
Let $G$ be a graph $P(n,2)$,
\begin{equation}
\gamma_{I}(G) =
\begin{cases}
\big\lceil\frac{4n}{5}\big\rceil      & n\equiv 0, 3, 4\ (\bmod\ 5),\\
\big\lceil\frac{4n}{5}\big\rceil+1    & n\equiv 1, 2\ (\bmod\ 5).
\end{cases}\notag
\end{equation}
\label{thm:Italianpn2}
\end{Theorem}

By Theorem \ref{thm:Italianpn2} and Lemma \ref{lem:tong2rainbowpn2}, we can obtain the following relation,
\begin{equation}
\begin{cases}
\gamma_{I}(P(n,2))=\gamma_{r_2}(P(n,2))      & n\equiv 0, 1, 2, 3, 4, 6, 7, 9(\bmod\ 10),\\
\gamma_{I}(P(n,2))=\gamma_{r_2}(P(n,2))-1    & n\equiv 5, 8(\bmod\ 10).
\end{cases}\notag
\end{equation}

\begin{Theorem}\emph{(\cite{Ebrahimi2009dominationpnk})}
If $n\geq5$, we have $\gamma(P(n, 2))=\lceil\frac{3n}{5}\rceil$.
\label{thm:dominationpn2}
\end{Theorem}

By Theorem \ref{thm:Italianpn2}-\ref{thm:dominationpn2}, $P(n, 2)$ is not Italian.

\section{The Italian domination number of $P(n,k)$ $(k\ge4)$}
\begin{Theorem}
For $k\ge4$,
\begin{equation}
\begin{cases}
\gamma_{I}(P(n,k))=\frac{4n}{5}     &k\equiv 2,3(\bmod\ 5), n\equiv 0(\bmod\ 5),\\
\frac{4n}{5}\leq\gamma_{I}(P(n,k))\leq \frac{4(n-k)}{5}\frac{(3k+2)}{(3k+1)}+\frac{4k+6}{3}  & otherwise.
\end{cases}\notag
\end{equation}
\label{thm:Italianpnk}
\end{Theorem}
\begin{proof}
By Lemma \ref{lem:pn2low}, $\gamma_{I}(P(n,k))\ge \frac{4n}{5}$. Then we will give the upper bounds by constructing some IDFs.

Case 1. $k\equiv0(\bmod5)$. First, we define a function $g$,
\begin{eqnarray*}
\setlength{\arraycolsep}{0pt}
g=
\left(\begin{array}{lllllll}
1&  0 & 1 & 0 & 0   \\
0&  0 & 0 & 1 & 1   \\
\end{array}\right)^\frac{k}{5}
\left(\begin{array}{lllll}
1  \\
0\\
\end{array}\right)
\left(\begin{array}{lllll}
0 & 1 & 0&  1 & 0\\
1 & 0 & 0&  0 & 1\\
\end{array}\right)^\frac{k}{5}
\left(\begin{array}{lllll}
1 & 0&  0&  1&  0 \\
0 & 1&  1&  0&  0 \\
\end{array}\right)^\frac{k}{5}
\left(\begin{array}{lllllll}
1&  0 & 1 & 0 & 0   \\
0&  0 & 0 & 1 & 1   \\
\end{array}\right)^\frac{k}{5}
\left(\begin{array}{lllll}
1  \\
1 \\
\end{array}\right)
\left(\begin{array}{lllll}
0 & 1 & 0&  1 & 0\\
1 & 0 & 0&  0 & 1\\
\end{array}\right)^\frac{k-5}{5}
\left(\begin{array}{lllll}
0 & 1 & 0  \\
1 & 0 & 0 \\
\end{array}\right).
\end{eqnarray*}

For $n\equiv0(\bmod\ 5k)$, we define $f$ by repeating $g$ $\frac{n}{5k}$ times, i.e. $f=g^{\frac{n}{5k}}$ or $f(v_i)=g(v_{i\bmod10k})$, then $w(f)=(4\times\frac{k}{5}\times 4+4\times\frac{k-5}{5}+5)\times \frac{n}{5k}=\frac{n(4k+1)}{5k}=\frac{4n}{5}\frac{(k+\frac{1}{4})}{k}$.

For $n\not\equiv0(\bmod\ 5k)$,
\begin{equation*}
f(v_{i})=\left\{
\begin{array}{llll}
g(v_{i\bmod10k})  &0\leq i \leq 2n-2k-1,\\
h(v_{i})          &2n-2k\leq i \leq 2n-1.
\end{array}\right.
\end{equation*}

\begin{equation*}
h(v_{ i })=\left\{
\begin{array}{llll}
1      &(i-(2n-2k))\equiv0, 1, 3, 5(\bmod6),\\
0      &(i-(2n-2k))\equiv2, 4(\bmod6).
\end{array}\right.
\end{equation*}

\noindent $h$ also can be expressed as follows.
\begin{equation*}
\setlength{\arraycolsep}{0pt}
h=\left\{
\begin{array}{llll}
\left(\begin{array}{lllll}
1 &0 &  0\\
1 &1 &  1\\
\end{array}\right)^\frac{k}{3}  &\hspace{5bp}k\equiv0(\bmod3),\\
\left(\begin{array}{lllll}
1 &0 &  0\\
1 &1 &  1\\
\end{array}\right)^\frac{k-1}{3}{\left(\begin{array}{lllll}
1  \\
1\\
\end{array}\right)}        &\hspace{5bp}k\equiv1(\bmod3),\\
{\left(\begin{array}{lllll}
1 &0 &  0\\
1 &1 &  1\\
\end{array}\right)}^\frac{k-2}{3}{\left(\begin{array}{lllll}
1 & 0 \\
1 & 1 \\
\end{array}\right)}       &\hspace{5bp}k\equiv2(\bmod3).
\end{array}\right.
\end{equation*}

\noindent The weight of $f$ is $w(f)=\lceil(4\times\frac{k}{5}\times 4+4\times\frac{k-5}{5}+5)\times \frac{n-k}{5k}\rceil+\lceil\frac{k}{3}\times4\rceil=\frac{4(n-k)}{5}\frac{(k+\frac{1}{4})}{k}+\frac{4k+6}{3}$.

Case 2. $k\equiv1(\bmod5)$. Let
\begin{eqnarray*}
\setlength{\arraycolsep}{0pt}
g=
\left(\begin{array}{lllllll}
1&  0 & 1 & 0 & 0   \\
0&  0 & 0 & 1 & 1   \\
\end{array}\right)^\frac{k-1}{5}
\left(\begin{array}{lllll}
1  \\
0\\
\end{array}\right)
\left(\begin{array}{lllll}
0 & 0 & 1&  0 & 1\\
1 & 1 & 0&  0 & 0\\
\end{array}\right)^\frac{k-1}{5}
\left(\begin{array}{lllll}
0 & 0&  1 \\
1 & 1&  0\\
\end{array}\right)
\left(\begin{array}{lllllll}
0&  0 & 1 & 0 & 1   \\
1&  1 & 0 & 0 & 0   \\
\end{array}\right)^\frac{k-1}{5}.
\end{eqnarray*}

For $n\equiv0(\bmod(3k+1))$, define $f=g^{\frac{n}{3k+1}}$ i.e. $f(v_i)=g(v_{i\bmod(6k+2)})$, then $w(f)=(4\times\frac{k-1}{5}\times 3+4)\times \frac{n}{3k+1}=\frac{4n}{5}\frac{(3k+2)}{(3k+1)}$.

For $n\not\equiv0(\bmod(3k+1))$,
\begin{equation*}
f(v_{i})=\left\{
\begin{array}{llll}
g(v_{i\bmod(6k+2)}) &0\leq i \leq 2n-2k-1,\\
h(v_{i})           &2n-2k\leq i \leq 2n-1.
\end{array}\right.
\end{equation*}
\noindent $h$ is defined in Case 1, then $w(f)=\lceil(4\times\frac{k-1}{5}\times3+4)\times\frac{n-k}{3k+1}\rceil +\lceil\frac{k}{3}\times4\rceil=\frac{4(n-k)}{5}\frac{(3k+2)}{(3k+1)}+\frac{4k+6}{3}$.

Case 3. $k\equiv2,3(\bmod5)$. We first define
\begin{eqnarray*}
\setlength{\arraycolsep}{0pt}
g=
\left(\begin{array}{lllllll}
1&  0 & 1 & 0 & 0   \\
0&  0 & 0 & 1 & 1
\end{array}\right).
\end{eqnarray*}

For $n\equiv0(\bmod5)$, let $f=g^{\frac{n}{5}}$ i.e. $f(v_i)=g(v_{i\bmod10})$, then $w(f)=4\times\frac{n}{5}=\frac{4n}{5}$.

For $n\not\equiv0(\bmod5)$,
\begin{equation*}
f(v_{i})=\left\{
\begin{array}{llll}
g(v_{i\bmod10})    &0\leq i \leq 2n-2k-1,\\
h(v_{i})           &2n-2k\leq i \leq 2n-1.
\end{array}\right.
\end{equation*}
\noindent $h$ is defined in Case 1, then $w(f)=\lceil4\times\frac{n-k}{5}\rceil+\lceil\frac{k}{3}\times4\rceil=\frac{4(n-k)}{5}+\frac{4k+6}{3}$.

Case 4. $k\equiv4(\bmod5)$. We define
\begin{eqnarray*}
\small{
\setlength{\arraycolsep}{0pt}
g=
\left(\begin{array}{lllllll}
1&  0 & 1 & 0 & 0   \\
0&  0 & 0 & 1 & 1   \\
\end{array}\right)^\frac{k+1}{5}
\left(\begin{array}{lllll}
1  \\
1\\
\end{array}\right)
\left(\begin{array}{lllll}
0 & 1 & 0&  1 & 0\\
1 & 0 & 0&  0 & 1\\
\end{array}\right)^\frac{k-4}{5}
\left(\begin{array}{lllll}
0 & 1 & 0&  1 \\
1 & 0 & 0&  0 \\
\end{array}\right)
\left(\begin{array}{lllll}
0 & 1&  0&  0&  1  \\
0 & 0&  1&  1&  0 \\
\end{array}\right)^\frac{k-4}{5}
\left(\begin{array}{lllll}
0 & 1&  0  \\
0 & 0&  1 \\
\end{array}\right)
\left(\begin{array}{lllllll}
1 & 0 & 0 & 1 & 0  \\
0 & 1 & 1 & 0 & 0\\
\end{array}\right)^\frac{k-4}{5}
\left(\begin{array}{lllll}
1 & 0 & 0 & 1 \\
0 & 1 & 1 & 0 \\
\end{array}\right)
\left(\begin{array}{lllll}
0 & 1 & 0&  1 & 0\\
1 & 0 & 0&  0 & 1\\
\end{array}\right)^\frac{k-4}{5}
\left(\begin{array}{lllll}
0 & 1 & 0  \\
1 & 0 & 0 \\
\end{array}\right).
}
\end{eqnarray*}

For $n\equiv0(\bmod\ 5k)$, let $f=g^{\frac{n}{5k}}$ i.e. $f(v_i)=g(v_{i\bmod10k})$, then $w(f)=(4\times\frac{k+1}{5}+4\times\frac{k-4}{5}\times 4+13)\times \frac{n}{5k}=\frac{4n}{5}\frac{(k+\frac{1}{4})}{k}$.

For $n\not\equiv0(\bmod\ 5k)$,
\begin{equation*}
f(v_{i})=\left\{
\begin{array}{llll}
g(v_{i\bmod10k}) &0\leq i \leq 2n-2k-1,\\
h(v_{i})         &2n-2k\leq i \leq 2n-1.
\end{array}\right.
\end{equation*}
\noindent $h$ is the same as in Case 1, then $w(f)=\lceil(4\times\frac{k+1}{5}+4\times\frac{k-4}{5}\times 4+13)\times\frac{n-k}{5k}\rceil+\lceil\frac{k}{3}\times 4\rceil=\frac{4(n-k)}{5}\frac{(k+\frac{1}{4})}{k}+\frac{4k+6}{3}.$
\end{proof}

\section{Conclusions}

The purpose of this paper is to study the Italian domination number of generalized Petersen graphs $P(n, k)$, $k\neq3$. We determine the exact values of $\gamma_I(P(n, 1))$, $\gamma_I(P(n, 2))$ and $\gamma_I(P(n, k))$ for $k\ge4$, $k\equiv2,3(\bmod5)$ and $n\equiv0(\bmod5)$. For other $P(n,k)$, we present a bound of $\gamma_I(P(n, k))$. We obtain the relationship between $\gamma_I(P(n,k))$ and $\gamma_{r2}(P(n,k))$ for $k=1,2$.
\begin{equation}
\begin{cases}
\gamma_{I}(P(n,1))=\gamma_{r_2}(P(n,1)) \\
\gamma_{I}(P(n,2))=\gamma_{r_2}(P(n,2))      & n\equiv 0, 1, 2, 3, 4, 6, 7, 9(\bmod\ 10),\\
\gamma_{I}(P(n,2))=\gamma_{r_2}(P(n,2))-1    & n\equiv 5, 8(\bmod\ 10).
\end{cases}\notag
\end{equation}

Moreover, our results imply $P(n,1)$ $(n\equiv0(\bmod\ 4))$ is Italian, $P(n,1)$ $(n\not\equiv0(\bmod\ 4))$ and $P(n,2)$ are not Italian.

\end{document}